\documentclass[12pt,a4paper]{amsart}
 \usepackage{amssymb,amsthm}
 \usepackage{multirow}
  \usepackage{mathtools}
 \usepackage{graphicx}
 \usepackage{booktabs}
 \newtheorem{thm}{Theorem}[section]

\newtheorem{prop}[thm]{Proposition}

\newtheorem{lem}[thm]{Lemma}

\newtheorem{rem}[thm]{Remark}
\newtheorem{defn}[thm]{Definition}

\newcommand{\say}[1]{``#1''}
\newcommand{\overbar}[1]{\mkern 1.5mu\overline{\mkern-1.5mu#1\mkern-1.5mu}\mkern 1.5mu}
 \begin{document}
 \title[Bounding the multiplicities in terms circuit rank]{Bounding the multiplicities of eigenvalues of graph matrices in terms of circuit rank using a new approach }
 \author{Ahmet Batal}
 \address{Department of Mathematics, Izmir Institute of Technology, 35430, Urla, Izmir, TURKEY}
\email{ahmetbatal@iyte.edu.tr}
\begin{abstract}
Let $G$ be a simple undirected graph, $\theta(G)$ be the circuit rank of $G$, $\eta_M(G)$ and $m_M(G,\lambda)$  be the nullity and the multiplicity of eigenvalue $\lambda$ of a graph matrix $M(G)$, respectively.  In the case $M(G)$ is the adjacency matrix $A(G)$, (the Laplacian matrix $L(G)$, the signless Laplacian matrix $Q(G)$) we find bounds to $m_M(G,\lambda)$ in terms of $\theta(G)$ when $\lambda$ is an integer (even integer, respectively). We also show that when $\alpha$ and $\lambda$ are rational numbers similar bounds can be found for $m_{A_{\alpha}}(G,\lambda)$ where $A_{\alpha}(G)$ is the generalized adjaceny matrix of $G$. Our bounds contain only $\theta(G)$, not a multiple of it. Up to now only bounds of $m_A(G,\lambda)$ (and later $m_{A_\alpha}(G,\lambda)$)  have been found in terms of the circuit rank and all of them contains $2\theta(G)$. There is only one exception in the case $\lambda=0$. Wong et al. (2022) showed that $\eta_A(G_c)\leq \theta(G_c)+1$, where $G_c$ is a connected cactus whose blocks are even cycles.  Our result, in particular, generalizes and extends this result to the multiplicity of any even eigenvalue of A(G) of any even connected graph $G$, and of any even eigenvalue of $L(G)$ and $Q(G)$ of any connected graph $G$. They also showed that $\eta_A(G_c)\leq 1$ when every block of the cactus is an odd cycle. This also corresponds a special case of our bound. We obtain these results by following a new approach. Instead of bounding $\eta_M(G)$ we bound $\text{\tiny{$2$}-}\eta_{\overbar{M}}(G)$ where $\overbar{M}$ is the binary matrix whose entries are equal to the entries of $M$ modulo $2$ and  $\text{\tiny{$2$}-}\eta_{\overbar{M}}(G)$ is the nullity of $\overbar{M}$ over the finite field $\mathbb{F}_2$. We show that the latter nullity over $\mathbb{F}_2$ dominates  not only $\eta_M(G)$, but also any $m_M(G,\lambda)$ for even $\lambda$. Then to connect the multiplicities to $\theta(G)$ we use the fact that $\text{\tiny{$2$}-}\eta_{\overbar{L}}(G)=\beta(G)+1$ for any connected graph $G$ where $\beta(G)$ is the dimension of the bicycle space which is a subspace of the cycle space of $G$.

\end{abstract}

 \maketitle

 \section*{Introduction}
Let $G$ be a simple undirected graph. We denote the number of vertices, edges, and connected components of $G$ by $\mathrm{v}(G)$, $\mathrm{e}(G)$, and $\mathrm{c}(G)$, respectively. Degree of a vertex is the number of its incident edges. We say $G$ is even (odd) if all of its vertices have even (odd) degree. We also denote the number of even degree vertices of $G$ by $\mathrm{v}_\mathrm{e}(G)$, number of odd degree vertices of $G$ by $\mathrm{v}_\mathrm{o}(G)$, and the number of even connected components of $G$ by $\mathrm{c}_\mathrm{e}(G)$.

 Let $V(G)=\{v_1,\cdots, v_n\}$ be the vertex set of $G$. Then the adjacency matrix  of $G$ is the  $n\times n$ symmetric matrix $A(G)=(a_{ij})$ where $a_{ij}=1$ if $v_i$ is adjacent to $v_j$ and $a_{ij}=0$, otherwise. $A(G)+I$ is called the closed neighborhood matrix of $G$, where $I$ is the $n\times n$  identity matrix. Let $D(G)$ be the $n\times n$ diagonal matrix whose ith diagonal entry is equal to the degree of $v_i$. Then, the Laplacian  matrix of $G$ is $L(G)=D(G)-A(G)$, and the signless Laplacian matrix of $G$ is $Q(G)=D(G)+A(G)$. We call all these matrices graph matrices of $G$. In general, we say $M(G)$ is a graph matrix of $G$ if $M(G)=\mathcal{D}(G)+tA(G)$  for some $n\times n$ diagonal matrix $\mathcal{D}(G)$ and $t\in \mathbb{R}$. We denote the multiplicity of an eigenvalue $\lambda$ of a graph matrix $M(G)$ by $m_M(G, \lambda)$. We denote the kernel space of a matrix $M$ over $\mathbb{R}$ by $Ker(M)$ and over $\mathbb{F}_2$ by $Ker_2(M)$. We denote their dimension by $\eta_M$ and $\text{\tiny{$2$}-}\eta_M$ and call them the nullity and $2$-nullity of $M$, respectively. In the case $M=M(G)$ is a graph matrix of $G$, we write $\eta_M=\eta_M(G)$ ($\text{\tiny{$2$}-}\eta_M=\text{\tiny{$2$}-}\eta_M(G)$) and call it the nullity ($2$-nullity) of $G$ with respect to $M(G)$. If $M(G)=A(G)$ we may just call it the nullity ($2$-nullity) of $G$. Clearly, $\eta_M(G)=m_M(G, 0)$ for all graph matrices $M$. Similarly, we denote the column space of a matrix $M$ over $\mathbb{R}$ by $Col(M)$, over $\mathbb{F}_2$ by $Col_2(M)$ and denote its rank
  over $\mathbb{R}$ by $r_M$, over $\mathbb{F}_2$ by $\text{\tiny{$2$}-}r_{M}$.

Let $E(G)=\{uv\;|\; \text{u is adjacent to v} \}$  be the edge set of $G$. Let us denote the collection of subsets of $E(G)$ by $Edge(G)$. It is a vector space over $\mathbb{F}_2$ where addition is taken as the symmetric difference operation of sets. 
Two well known subspaces of $Edge(G)$ are the cycle and cut spaces. The cycle space $Cycle(G)$ consists of even subgraphs of $G$.  We call the dimension $\theta(G)$ of $Cycle(G)$ the circuit rank of $G$. It satisfies
$$\theta(G)= (\text{e}-\text{v}+\text{c})(G).$$
On the other hand, every subset of vertices $S$ of $V(G)$ determines a cut set of edges $\{uv\in E(G) |\; u\in S,\; v\in S^c\},$ and the collection of all cut sets (or spanning cut subgraphs) $Cut(G)$ is the orthogonal complement of $Cycle(G)$. The intersection of $Cycle(G)$  and $Cut(G)$ is the bicycle space $Bicycle(G)$ and we denote its dimension by $\beta(G)$.

Bounding the nullity of any graph (or a specific graph type) using some graph parameters and characterizing the graphs whose nullities attain the given bound is an active area of research of graph theory (See \cite{MR2775770}, \cite{MR2915279}, \cite{MR3825702}, \cite{MR2775771}, \cite{MR2997803}, \cite{MR2116461}, \cite{MR2680264}, \cite{MR2859917}, \cite{MR2547911}, \cite{MR1823617}, \cite{MR2444330}, \cite{MR2950470}, \cite{MR2889577}, \cite{MR2312029}, \cite{MR3722834}, \cite{MR3828816}, \cite{MR3462998}, \cite{MR2506874}, \cite{MR1600771}, \cite{MR3274682}, \cite{MR3479403}, \cite{MR3459054}, \cite{MR2997815}, \cite{MR2166864}, \cite{MR3135927}, \cite{MR3834206}). In particular, finding bounds to $\eta_A(G)$ as well as $m_A(G, \lambda)$ in terms of $\theta(G)$  plus some additional graph parameters gained attention in recent years (See \cite{MR4078896}, \cite{MR4377152}, \cite{Wei20}, \cite{MR3848106}, \cite{Wong16}, \cite{Wang20}, \cite{Wong314}, \cite{MR4506593}, \cite{Wong22}, \cite{WongwithZhou22}, \cite{MR4069031}). One of the pioneering work in this direction is the work of Ma, Wong, and Tian \cite{Wong16}, in which they showed that for all graphs $G$ which does not contain isolated vertices,
\begin{equation}
\label{ineq2tp}
\eta_A(G)\leq (2\theta+ p)(G),
\end{equation} where $p(G)$ is the number of pendant (degree one) vertices of $G$. Then the generalization of this bound to any $m_A(G, \lambda)$ is obtained by Wang et al. in \cite{Wang20}.
However, considering the facts that $\mathrm{v}(G)$ is the trivial bound of  $\eta_A(G)$, and $2\theta(G)$ involves the term $2\mathrm{e}(G)-2\mathrm{v}(G)$, it is clear that this estimate is not applicable if $\mathrm{e}(G)\geq 3\mathrm{v}(G)/2$. So it would be better if we can replace the term $2\theta(G)$ by $\theta(G)$ in the bound of $\eta_A(G)$ at the cost of some additional terms if necessary. Although this new estimate may not be stronger than \eqref{ineq2tp} for all graph types, it would still be applicable to important graph types where \eqref{ineq2tp} is not applicable. As far as we know, the only result which gives a bound of  $\eta_A(G)$ involving $\theta(G)$ instead of $2\theta(G)$ was obtained in a recent paper \cite{Wong22} of Wong et al. However, they have been able to obtain this bound by putting strict assumptions on $G$. More precisely, they showed that if $G$ is a bipartite graph consisting of cycle blocks (thus, necessarily an even graph) then
\begin{equation}
\label{ineqtc}
\eta_A(G)\leq (\theta+ \mathrm{c})(G).
\end{equation}
As we will see in a moment these assumptions are not necessary. Indeed, \eqref{ineqtc} holds true for all even graphs $G$.

It seems like the classic approach which has been used up to now by the above authors has its own limits which makes it very hard if not impossible to obtain bounds of $\eta_A(G)$ involving $\theta(G)$ instead of $2\theta(G)$ for an arbitrary graph $G$. By classic approach we mean that the methods which use main inequalities satisfied by $\eta_A(G)$ under some graph operations and graph theoretical arguments which connects $\eta_A(G)$ to $\theta(G)$.

In this paper we introduce a new approach to connect $\eta_A(G)$ to $\theta(G)$  which we think more natural and also shorter than the classical one. The idea lies in the following simple observation:
$$\eta_A(G)\leq\text{\tiny{$2$}-}\eta_A(G)$$
for all graphs $G$ (See Lemma \ref{lemnull}). So instead of bounding $\eta_A(G)$ directly, we can try to bound $\text{\tiny{$2$}-}\eta_A(G)$ by $\theta(G)$ instead. This sounds more natural because both $\text{\tiny{$2$}-}\eta_A(G)$ and $\theta(G)$ are dimensions of linear spaces over the same field $\mathbb{F}_2$, namely $Ker_2(A(G))$ and $Cycle(G)$. So one can connect $\text{\tiny{$2$}-}\eta_A(G)$ and $\theta(G)$  more easily if one can find a linear algebraic correspondence between these spaces (or between spaces related to them). Indeed there is such a correspondence in the case $G$ is an even graph. To explain this let us write $\overbar{m}:= m\mod{2}$ for an integer $m$ and $\overbar{M}:=(\overbar{m_{ij}})$ for an integer valued matrix $M= (m_{ij})$. Then $Ker_2(A(G))$ is equal to $Ker_2(\overbar{L}(G))$ if $G$ is even and a well known connection (\cite[Lemma 14.15.3]{GodsilRoyle01}) between the latter space and the bicycle space gives us
\begin{equation}
\label{ineqtc2}
\text{\tiny{$2$}-}\eta_{A}(G)= (\beta+\text{c})(G)\leq (\theta+\text{c})(G),
\end{equation}
when $G$ is an even graph. Hence we immediately obtain \eqref{ineqtc} for any even graph $G$ using this approach.
Moreover, we can generalize \eqref{ineqtc2} to any graph by a simple argument at the cost of adding the term $\mathrm{v}_\mathrm{o}(G)$ to the right hand side (See Theorem \ref{thmev}). This shows how powerful and practical following this approach might be considering the fact that obtaining \eqref{ineqtc} even only for very special graphs takes effort when we follow the classical approach.

More importantly, applicability of this approach is not restricted to finding bounds of $\eta_A(G)$. Indeed, $\text{\tiny{$2$}-}\eta_A(G)$  bounds not only $\eta_A(G)$ but also any  $m_A(G, \lambda)$ when $\lambda$ is an even integer (See Lemma \ref{lemnull}). On the other hand, for odd eigenvalues $\lambda$, we have $m_A(G, \lambda) \leq \eta_{A+I}(G)$, where the latter nullity of $G$ with respect to $A(G)+I$ can be bounded by $(\theta+\text{c}+\mathrm{v}_\mathrm{e})(G)$ (See Theorem \ref{thmev}).

Adding the terms $\mathrm{v}_\mathrm{e}(G)$ or $\mathrm{v}_\mathrm{o}(G)$ to the corresponding bounds in the general case seems inevitable but there is a graph parameter that we call $\tau(G)$ which we can subtract from these bounds as well. To define it, let us say a cycle is even (odd) if it has even (odd) number of edges and two cycles are disjoint if they do not share any edge. Then we define $\tau(G)$ as the maximum of the cardinalities of the sets of disjoint odd cycles of $G$. The reason why we can subtract $\tau(G)$ from the bounds is because of the fact that $\tau(G)$ corresponds to the dimension of a subspace of $Cycle(G)$  which intersects $Bicycle(G)$ only trivially (See Lemma \ref{lemtau} and the discussion above it).

Similarly, the same approach can be directly used to estimate the multiplicities of the integer eigenvalues of $A(G)+I$, $L(G)$ and $Q(G)$ as well. More precisely, we have the following theorem as our main result.

\begin{thm}
\label{thethm} Let $G$ be a graph and $\lambda$ be an integer eigenvalue of the corresponding graph matrix. Then
\begin{align}
m_A(G, \lambda) &\leq (\theta-\tau+\mathrm{v}_\mathrm{o}+2\mathrm{c}_\mathrm{e}-\mathrm{c}+2i)(G) &                       &\text{if $\lambda$ is even},\nonumber\\
m_A(G, \lambda) &\leq (\theta-\tau+\mathrm{v}_\mathrm{e}+\mathrm{c})(G) &                     &\text{if $\lambda$ is odd},\nonumber\\
m_L(G, \lambda), m_Q(G, \lambda)  &\leq (\theta-\tau+\mathrm{c})(G) &   &\text{if $\lambda$ is even}, \nonumber \\
m_L(G, \lambda), m_Q(G, \lambda) &\leq (\mathrm{v}-\beta-\mathrm{c})(G)&   &\text{if $\lambda$ is odd}, \nonumber
\end{align}
where $i(G)=0$ if $G$ is even and $i(G)=1$ otherwise.
\end{thm}
The proof is given in Section $1$.
Note that for $m_L(G, \lambda)$ and  $m_Q(G, \lambda)$ when $\lambda$ is odd we cannot obtain a bound  in terms of $\theta(G)$ unlike the even case. This is something expected since it is known that $m_L(G, 1)\geq (p-q)(G)$ where $q(G)$ is the number of vertices adjacent to a pendant vertex of $G$ \cite{Faria85}. Hence, even if $\theta(G)$ is small, $m_L(G, \lambda)$ can be very large for odd $\lambda$. It might be the case that the multiplicities of odd eigenvalues are bounded in terms of $\theta(G)$ with some additional terms involving at least $(p-q)(G)$. However, such a bound, if it exists, cannot be found using solely our approach, it must involve some other arguments as well.

We already mentioned that Wang et al. in \cite{Wang20} generalized \eqref{ineq2tp} to any $m_A(G, \lambda)$. Later Shuchao and Wei in \cite{Wei20} generalized the latter result further to any $m_{A_\alpha}(G, \lambda)$, where $A_\alpha(G)$ is the graph matrix defined by $A_\alpha(G):=\alpha D(G) +(1-\alpha) A(G)$ where $0\leq \alpha \leq 1.$ $A_\alpha(G)$ matrix was first introduced by Nikiforov in \cite{Nikiforov17} and its spectral analysis gained a lot of attention later on. We show that our approach is applicable to this graph matrix as well in the case where both $\alpha$ and $\lambda$ are rational numbers. The bound of $m_{A_\alpha}(G, \lambda)$ changes drastically depending on the representations of $\alpha$ and $\lambda$ in their lowest forms as ratio of two integers. The results are summarized Section $2$ in Theorem \ref{thmalpha}.

In the last section we give examples of some types of graphs attaining the corresponding bounds in Theorem \ref{thethm}.

\section{Proof of Theorem \ref{thethm}}

\begin{lem}
\label{lemnull}
Let $M$ be an integer valued matrix. Then $$\eta_M\leq \text{\tiny{$2$}-}\eta_{\overbar{M}}.$$
Moreover, if $M$ is a graph matrix of $G$, then
\begin{align}
\quad \quad \quad\quad \quad \quad m_M(G, \lambda)&\leq \text{\tiny{$2$}-}\eta_{\overbar{M}}(G)           &          &\text{if $\lambda$ is an even integer},\nonumber\\
\quad \quad \quad\quad \quad \quad m_M(G, \lambda)&\leq \text{\tiny{$2$}-}\eta_{\overbar{M}+I}(G)         &          &\text{if $\lambda$ is an odd integer}.\nonumber
\end{align}
\end{lem}
\begin{proof}

For the proof of the first inequality, we will show that $r_M \geq \text{\tiny{$2$}-}r_{\overbar{M}}$. Then the result follows by the rank nullity theorem.

Take a set of column vectors $\{c_1,\dots, c_k\}$ of $M$  such that $\{\overbar{c_1},\dots, \overbar{c_k}\}$ is a linearly independent set in $\mathbb{F}_2^n$ where $n$ is the number of rows of $M$. Consider the following equation
\begin{equation}
\label{eqindep}
\alpha_1c_1+\cdots + \alpha_kc_k=0,
\end{equation}
where $\alpha_i \in \mathbb{R}$ for $i=1,\dots, k$. Since each $c_i$ has integer entries, \eqref{eqindep} has a nontrivial solution in $\mathbb{R}$ if and only if it has a nontrivial solution in $\mathbb{Z}$. So without loss of generality, we can assume all $\alpha_i$'s are integer. Dividing \eqref{eqindep} by a suitable power of $2$ if necessary, we can further assume that all $\alpha_i$'s are integers such that not all of them are even unless all of them are zero. Then the corresponding column vectors in $\mathbb{F}_2^n$ satisfy

\begin{equation}
\label{eqindep2}
\overbar{\alpha_1}\overbar{c_1}+\cdots + \overbar{\alpha_k}\overbar{c_k}=0,
\end{equation}
which has only the trivial solution since $\{\overbar{c_1},\dots, \overbar{c_k}\}$ is linearly independent. This means every $\alpha_i$ is even and this is only possible if every $\alpha_i$ is zero by our assumption. Hence, $\{c_1,\dots, c_k\}$ is a linearly independent set of $\mathbb{R}^n$. This proves $r_M \geq \text{\tiny{$2$}-}r_{\overbar{M}}$.

For the last two inequalities just note that
$$ m_M(G, \lambda)=\eta_{M-\lambda I}(G) \leq \text{\tiny{$2$}-}\eta_{\overbar{M-\lambda I}}(G)= \text{\tiny{$2$}-}\eta_{\overbar{M}+\overbar{\lambda}I}(G).$$
\end{proof}

Let $P(G):=\overbar{L}(G)=\overbar{Q}(G)$. We call it the parity Laplacian matrix of $G$. We have the following elementary fact which can be found in \cite[Lemma 14.15.3]{GodsilRoyle01}:

 \begin{prop}$($\cite{GodsilRoyle01}$)$
 \label{propmain}
 For any graph $G$
\begin{equation}
 \text{\tiny{$2$}-}\eta_{P}(G)= (\beta+\mathrm{c})(G).
\end{equation}
\end{prop}


Since the bicycle space is a subspace of the cycle space, the above lemma immediately implies
\begin{equation}
\label{inqpre}
\text{\tiny{$2$}-}\eta_{P}(G)\leq (\theta+\mathrm{c})(G).
\end{equation}
However this inequality can be improved. Note that, in general, $Bicycle(G)$ is a proper subspace of $Cycle(G)$. It consists of those even subgraphs of $G$ which are cuts as well. Let us call these subgraphs even cuts of $G$. Since $Cycle(G)$ and $Cut(G)$ are orthogonal to each other under the dot product \say{$\cdot$}, defined by $A\cdot B:= |A \cap B| \mod 2$, a necessary condition to be an even cut is to be orthogonal to every even subgraph. In other words, an even cut must share even number of edges with every even subgraph including itself. So trivially, no odd cycle or no subgraph including an odd cycle can be an even cut. This observation leads us to the following lemma.

\begin{lem}
\label{lemtau}
For any graph $G$, $$\beta(G) \leq (\theta-\tau)(G).$$
\end{lem}
\begin{proof}
Let $\mathcal{T}$ be a set of disjoint odd cycles of $G$ with $|\mathcal{T}|=\tau(G)$. Note that the span of $\mathcal{T}$, $Span(\mathcal{T})$, is a subspace of $Cycle(G)$ with dimension $\tau(G)$ and $Bicycle(G)\cap Span(\mathcal{T}) =\{\emptyset\}$, where $\emptyset$ denotes the empty spanning subgraph of $G$. Therefore, $ Bicycle(G)+Span(\mathcal{T})$  is a subspace of $Cycle(G)$ with dimension $(\beta+\tau)(G)$. Hence, $(\beta+\tau)(G) \leq \theta(G)$.
\end{proof}
Proposition \ref{propmain} and Lemma \ref{lemtau} give us the following.
\begin{thm}
\label{thmmain}
For any graph $G$
\begin{equation}
\label{inqmain}
\text{\tiny{$2$}-}\eta_{P}(G)\leq (\theta-\tau+\mathrm{c})(G).
\end{equation}
\end{thm}

Note that $P(G)=\overbar{D}(G)+A(G)$ where $\overbar{D}$ is the diagonal matrix whose diagonal entries are the parity of the degrees of the corresponding vertices. Hence, $P(G)=A(G)$ if $G$ is even, and $P(G)=A(G)+I$ if $G$ is odd. Therefore, from \eqref{inqmain} we obtain
 \begin{align}
 \label{inqeven} \text{\tiny{$2$}-}\eta_{A}(G)&\leq (\theta-\tau+\text{c})(G) & &\text{if $G$ is even},\\
 \label{inqodd} \text{\tiny{$2$}-}\eta_{A+I}(G) &\leq (\theta-\tau+\text{c})(G) & &\text{if $G$ is odd}.
\end{align}

\begin{rem}
 Recently, Wong et al. in \cite{Wong22} proved that $\eta_A(G) \leq \theta(G) +1$ for bipartite, cycle spliced graphs (connected even graphs consisting of cycle blocks). On the other hand, Lemma \ref{lemnull} and \eqref{inqeven} show us that this is true for any connected even graph. They also showed in \cite[Theorem 1.2]{Wong22} that if all cycles of a cycle spliced graph are odd, then $\eta_A(G)\leq 1$. This fact can also be realized by our approach.
 Indeed, in such a case $\theta(G)=\tau(G)$. Therefore, Lemma \ref{lemnull}, \eqref{inqeven} and connectedness of $G$  imply $\eta_A(G) \leq 1$.
\end{rem}

In the view of Lemma \ref{lemnull}, we will use \eqref{inqmain}, \eqref{inqeven}, and \eqref{inqodd} to estimate not only the nullities but in general the multiplicities of the integer eigenvalues of the adjacency, the Laplacian, and the signless Laplacian matrices of $G$. However, since \eqref{inqeven} and \eqref{inqodd} are true only for even and odd graphs, respectively, to be able to benefit from them for an arbitrary graph $G$, we construct an even and an odd graph using $G$ whose graph parameters are related to the graph parameters of $G$ in a suitable way.

\begin{defn}
For a graph $G$ we define ${G_{even}}$ as $G$ itsef if $G$ is even and otherwise as the graph obtained from $G$ by connecting every odd degree vertex of $G$ to a unique external vertex by an edge. We define ${G_{odd}}$ as the graph obtained from $G$ by adding a pendant edge to every even degree vertex of $G$.
\end{defn}

It is clear that $G_{odd}$ is an odd graph. Also $G_{even}$ is an even graph by the handshaking lemma. Let $i(G)$ be the binary valued function which takes the value $0$ if $G$ is even and $1$ if $G$ is odd. Then we have the following lemma.

\begin{lem}
\label{lem8}Let $G$ be a graph and $M$ be a binary valued graph matrix of $G$. Then the followings hold.
\begin{align}
i)\;  &\text{\tiny{$2$}-}\eta_M(G_{even})\geq (\text{\tiny{$2$}-}\eta_M-i)(G),     &v)\; &\text{\tiny{$2$}-}\eta_M(G_{odd})\geq (\text{\tiny{$2$}-}\eta_M-\mathrm{v}_\mathrm{e})(G),  \nonumber\\
ii)\; &\theta(G_{even})=(\theta+\mathrm{v}_\mathrm{o}+\mathrm{c}_\mathrm{e}-\mathrm{c})(G),   &vi)\; &\theta(G_{odd})=\theta(G), \nonumber\\
iii)\; &\mathrm{c}(G_{even})= (\mathrm{c}_\mathrm{e}+i)(G),      &vii)\; &\mathrm{c}(G_{odd})= \mathrm{c}(G),  \nonumber\\
iv)\; &\tau(G_{even})\geq \tau(G),   &viii)\; &\tau(G_{odd})= \tau(G). \nonumber
\end{align}
\end{lem}
\begin{proof}
We will only prove (i), (ii), (iii), and (v) since the rest are obvious.

For (i) and (v) let $H-v$ denote the induced subgraph obtained from a graph $H$ by removing one of its vertices $v$. Then the binary graph matrix $M(H-v)$  is the principal submatrix of $M(H)$ obtained by deleting the row and column of $M(H)$ corresponding to the vertex $v$. It is an elementary fact that in such a case $\text{\tiny{$2$}-}\eta_{M}(H-v)\leq \text{\tiny{$2$}-}\eta_{M}(H)+1$. If $G$ is even then $G_{even}=G$ and there is nothing to prove. Otherwise, $G=G_{even}-u$, where $u$ is the external vertex we added to $G$ to obtain $G_{even}$. Hence (i) immediately follows from this inequality. Also the same inequality proves (v) since $G$ can be obtained from $G_{odd}$ by removing $\mathrm{v}_\mathrm{e}(G)$ number of pendant vertices of $G_{odd}$ one by one.

(iii) follows from the fact that not even components of $G$ become connected to each other in $G_{even}$ via the paths of length two passing trough the external vertex.

(ii) follows from combining the facts that $\mathrm{e}(G_{even})= (\mathrm{e}+\mathrm{v}_\mathrm{o})(G)$, $\mathrm{v}(G_{even})=(\mathrm{v}+i)(G)$, and $\mathrm{c}(G_{even})= (\mathrm{c}_\mathrm{e}+i)(G)$.
\end{proof}

\begin{thm} \label{thmev} Let $G$ be a graph. Then
\begin{align}
\label{inqa}\text{\tiny{$2$}-}\eta_{A}(G)&\leq (\theta-\tau+\mathrm{v}_\mathrm{o}+2\mathrm{c}_\mathrm{e}-\mathrm{c}+2i)(G),\\
\label{inqai}\text{\tiny{$2$}-}\eta_{A+I}(G)&\leq (\theta-\tau+\mathrm{v}_\mathrm{e}+\mathrm{c})(G),
\end{align}
\end{thm}
\begin{proof}
\eqref{inqeven} and Lemma \ref{lem8} imply
\begin{align}
\text{\tiny{$2$}-}\eta_{A}(G)&\leq \text{\tiny{$2$}-}\eta_A(G_{even})+i(G) \nonumber \\
                             &\leq (\theta-\tau+\text{c})(G_{even})+i(G) \nonumber \\
                             &\leq (\theta+\mathrm{v}_\mathrm{o}+\mathrm{c}_\mathrm{e}-\mathrm{c}-\tau +\mathrm{c}_\mathrm{e}+i)(G)+i(G) \nonumber \\
                             &=(\theta-\tau+\mathrm{v}_\mathrm{o}+2\mathrm{c}_\mathrm{e}-\mathrm{c}+2i)(G)
\end{align}
and
\begin{align}
\text{\tiny{$2$}-}\eta_{A+I}(G)&\leq \text{\tiny{$2$}-}\eta_A(G_{odd})+\mathrm{v}_\mathrm{e}(G) \nonumber \\
                               &\leq (\theta-\tau+\text{c})(G_{odd})+\mathrm{v}_\mathrm{e}(G)  \nonumber \\
                               &= (\theta-\tau+\mathrm{v}_\mathrm{e}+\mathrm{c})(G). \nonumber
\end{align}
\end{proof}

Theorem \ref{thmmain} and Theorem \ref{thmev}, together with Lemma \ref{lemnull}, complete the proof of Theorem \ref{thethm} for the first three bounds. On the other hand, the following theorem with Lemma \ref{lemnull} gives us the last bound.

\begin{thm}
\label{thmtirt}
For any graph $G$
\begin{equation}
\label{inqtirt}
\text{\tiny{$2$}-}\eta_{P+I}(G)\leq (\mathrm{v}-\beta-\mathrm{c})(G).
\end{equation}
\end{thm}
\begin{proof}
First note that $Ker_2(P(G)+I)\subseteq Col_2(P(G))$. Indeed,\\ $(P(G)+I)x=0$ is equivalent to say that $P(G)x=x$. Therefore, together with the rank nullity theorem and Proposition \ref{propmain}
we have
\begin{align}
\text{\tiny{$2$}-}\eta_{P+I}(G)&\leq \text{\tiny{$2$}-}r_{P}(G) \nonumber \\
                               &=(\mathrm{v}-\text{\tiny{$2$}-}\eta_P)(G)\nonumber  \\
                               &= (\mathrm{v}-\beta-\mathrm{c})(G). \nonumber
\end{align}
\end{proof}


\section{Multiplicities of the eigenvalues of $A_{\alpha}$}
Nikiforov introduced the graph matrix $A_\alpha(G)$ in \cite{Nikiforov17} (See introduction for the definition). It is called the generalized adjacency matrix of $G$. Since $A_0(G)=A(G)$ and $2A_{1/2}(G)=Q(G)$, the spectral analysis of  $A_\alpha(G)$ combines the spectral analysis of $A(G)$ and $Q(G)$. Therefore, later on the spectral properties of $A_\alpha(G)$ have been studied by lots of researcher among which  Shuchao and Wei in \cite{Wei20} showed that $m_{A_\alpha}(G, \lambda) \leq (2\theta+p)(G)$ for all $\alpha\in[0,1]$, and for all $\lambda\in \mathbb{R}$. Hence they generalized the results in \cite{Wong16} and \cite{Wang20} to the multiplicities of the eigenvalues of $A_\alpha(G)$. Our method can also be applied to this graph matrix in the case $\alpha$ and $\lambda$ are rational numbers. Indeed, we obtain the following results.

\begin{thm}
\label{thmalpha}
Let $G$ be a graph and $a/b$ and $p/q$ be rational numbers in their lowest forms. Then we have the following table.
\begin{table}[h]
\label{table1}
\caption{ Bounds of $m_{A_{a/b}}(G, p/q)$} 
\centering 
\begin{tabular}{|c|c|c|c|l|} 
\hline\hline 
\small  $a$ & \small $b$ & \small $p$  & \small $q$   & \small { $m_{A_{a/b}}(G, p/q)$ is less than or equal to} \\ \hline
odd & odd & odd & odd &   $\text{\tiny{$2$}-}\eta_{\overbar{D}+I}(G)= \mathrm{v}_\mathrm{o}(G)$  \\ [0.3ex] \hline
even & odd & odd & odd &  $\text{\tiny{$2$}-}\eta_{A+I}(G)\leq (\theta-\tau+\mathrm{v}_\mathrm{e}+\mathrm{c})(G) $ \\  [0.3ex]\hline
odd & odd & even & odd &  $\text{\tiny{$2$}-}\eta_{\overbar{D}}(G)=\mathrm{v}_\mathrm{e}(G)$  \\ [0.3ex]\hline
even & odd & even & odd &  $\text{\tiny{$2$}-}\eta_{A}(G)\leq (\theta-\tau+\mathrm{v}_\mathrm{o}+2\mathrm{c}_\mathrm{e}-\mathrm{c}+2i)(G)$  \\ [0.3ex]\hline
odd & even & - & odd &  $\text{\tiny{$2$}-}\eta_{P}(G)\leq (\theta-\tau+\mathrm{c})(G) $  \\\hline
  - & odd & odd & even &   $\text{\tiny{$2$}-}\eta_{I}=0 $  \\ [0.1ex]\hline
\end{tabular}
\label{table1}
\end{table}
\end{thm}
\begin{proof}
Note that
\begin{align}
 Ker(A_{a/b}-(p/q)I)&= Ker(qA_{a/b}-pI)        \nonumber\\
                  &= Ker((qa/b)D+q(1-a/b)A-pI) \nonumber\\
                  &= Ker(qaD+q(b-a)A-bpI). \nonumber
\end{align}
So if we let $M= qaD+q(b-a)A-bpI$, then we have $$m_{A_{a/b}}(G, p/q)=\eta_{(A_{a/b}-(p/q)I)}= \eta_M \leq \text{\tiny{$2$}-}\eta_{\overbar{M}} $$
by Lemma \ref{lemnull}.
On the other hand, $$\overbar{M}= \overbar{q}\overbar{a}\overbar{D}+\overbar{q}(\overbar{b}+\overbar{a})A+\overbar{b}\overbar{p}I.$$
Now the results follow by finding the special form of $\overbar{M}$ in each case and applying Theorems \ref{thmmain}, \ref{thmev}, and \ref{thmtirt}.
\end{proof}

\begin{rem}
Our method in the not included case in Theorem \ref{thmalpha} does not give a nontrivial bound to the corresponding multiplicity.
\end{rem}

\section{Existence of the extremal graphs}

Before giving characterizations of the extremal graphs whose eigenvalue multiplicities attain the corresponding bound in Theorem \ref{thethm}, we think that one should first examine the extremal graphs whose $2$-nullities of the corresponding graph matrix attain the bound given in Theorems \ref{thmmain}, \ref{thmev}, and \ref{thmtirt} since the bounds in the first theorem were obtained by the bounds of the $2$-nullities in the latter theorems by directly applying Lemma \ref{lemnull}. On the other hand, this analysis seems to require a comprehensive study. Note that there are four graph matrices whose $2$-nullities need to be examined, namely $A(G)$, $A(G)+I$, $P(G)$, and $P(G)+I$. Unfortunately, as far as we know, $2$-nullities of the corresponding matrices is not a well studied area and their properties such as how they change under certain graph operations are not examined enough yet. The one exception might be the $2$-nullity of $A+I$. Because of its connection with the Lights Out game and also with the parity domination theory some of its properties have been examined in several studies such as \cite{Amin92}, \cite{Amin96}, \cite{Amin98}, \cite{Amin02}, \cite{Ballard19} . So we hope that after this study, there will be more attraction to the analysis of the $2$-nullities of the corresponding graph matrices by other researchers as well.

One thing we can say for sure for now is that all of the bounds in Theorem \ref{thethm}, an hence in Theorems \ref{thmmain}, \ref{thmev}, and \ref{thmtirt} are attained. For example it is easy to check that for each $k$, the cycle $C_{k}$ is an extremal graph attaining the bounds in \eqref{inqmain} and \eqref{inqa}. Moreover, $m_A(C_{4k}, 0)= m_L(C_{4k}, 2)= m_Q(C_{4k}, 2)=2$ which is equal to the corresponding bounds in the first and third inequalities in Theorem \ref{thethm}.

On the other hand, any odd tree $T$ is an extremal graph attaining the bound in \eqref{inqai}. Indeed, $\text{\tiny{$2$}-}\eta_{A+I}(T)\leq 1$ by  \eqref{inqai}. On the other hand, the vector $(1,\dots,1)^t$ is in $Ker_2(A(T)+I)$ since $T$ is odd, hence $\text{\tiny{$2$}-}\eta_{A+I}(T)=1$. Moreover, it is not hard to see that for the odd tree $H$, which is in the shape of the letter H, we have $m_A(H,1)=1$ which is the bound we obtain using the second inequality in Theorem \ref{thethm}.

Lastly, any complete graph $K_{2k+1}$ with $2k+1$ vertices is an extremal graph attaining the bound in \eqref{inqtirt}. Indeed $P(K_{2k+1})+I$ is the square matrix all of whose entries are $1$. Therefore, $\text{\tiny{$2$}-}\eta_{P+I}(K_{2k+1})= 2k$ which is the largest possible bound we can obtain using \eqref{inqtirt}. Moreover, $m_L(K_{2k+1},2k+1)=m_Q(K_{2k+1},2k-1)=2k$. Hence, the bound in the last inequality of Theorem \ref{thethm} is also attained.

\textbf{Declaration of Competing Interests} The author declares that he has no
known competing financial interests or personal relationships that could have appeared to influence the work reported in this paper. 

\end{document}